\documentclass[11pt]{amsart}
\usepackage[utf8x]{inputenc}
\usepackage{cmap}
\usepackage{geometry} 
\usepackage{graphicx}
\usepackage{epstopdf}
\usepackage{hyperref}
\graphicspath{{images/}}

\newcommand{\ceil}[1]{\left\lceil #1 \right\rceil}

\newtheorem{theorem}{Theorem}[section]

\newtheorem{proposition}{Proposition}[section]
\newtheorem{conjecture}{Conjecture}[section]
\newtheorem{corollary}{Corollary}[section]

\begin{document}
\title[On correlation of hyperbolic volumes]{On correlation of hyperbolic volumes of fullerenes with their properties}
\author{A.~Egorov, A.~Vesnin}
\address{Novosibirsk State University, Russia} 
\email{a.egorov2@g.nsu.ru} 
\address{Tomsk State University and Sobolev Institute of Mathematics, Russia} 
\email{vesnin@math.nsc.ru}

\thanks{The work was supported in part by the Theoretical Physics and Mathematics Advancement Foundation ``BASIS''} 

\begin{abstract} 
We observe that fullerene graphs are one-skeletons of polyhedra, which can be realized with all dihedral angles equal to $\pi/2$ in a hyperbolic 3-dimensional space. One of the most important invariants of such a polyhedron is its volume. We are referring this volume as a hyperbolic volume of a fullerene. It is known that some topological indices of graphs of chemical compounds serve as strong descriptors and correlate with chemical properties. We demonstrate that  hyperbolic volume of fullerenes correlates with few important topological indices and so, hyperbolic volume  can serve as a chemical descriptor too. The correlation between hyperbolic volume of fullerene and its Wiener index suggested few conjectures on  volumes of hyperbolic polyhedra. These conjectures  are confirmed for the initial list of fullerenes. 
\end{abstract}
		
\maketitle 

\section{Introduction} \label{sec1} 

Mathematical methods became a very strong research tool in biology and chemistry in the last decades. We refer to~\cite{Xia} for a modern survey on geometric, topological, and graph theoretical methods for the modeling and analysis of biomolecular data and to~\cite{Lo} for the classification of common chemical descriptors for QSAR/QSPR analysis used for machine learning in chemical informatics and drug discovery. In this paper, we will consider 2D-Chemical Descriptors based on Structural topology, such as the Wiener index. We will discuss the properties of fullerenes. Fullerenes are one of the most famous classes of chemical structures investigated by geometric and topological methods in recent years, see, for example,~\cite{CataldoBook, AshrafiBook}. 

We denote by $C_{N}$ a fullerene with $N$ carbon atoms. It is known that $C_{N}$ has many isomers. We will be interested in ordering of fullerenes as well as in ordering of fullerenes with a fixed number of carbons. One of the approaches for ordering is based on graph-theoretical properties of fullerenes. Namely, some topological indices can be used, say the Wiener index, the pentagon signature, etc. Another approach, which we present in this paper, is based on non-Euclidean geometry invariants. We will show that any fullerene can be realized in a hyperbolic 3-dimensional space as a polyhedron with all dihedral angles equals $\pi/2$, see~\cite{Ratcliffe} for more information about hyperbolic geometry. It is known that such a realization is unique up to isometry. Therefore, hyperbolic volumes of fullerenes can be taken as a good parameter for the ordering of them.  

The paper is organized as follows. 
We begin in Section~\ref{sec2} by reminding some basic properties of fullerenes and mathematical propertied of fullerene graphs. In Section~\ref{sec3} we discuss a hyperbolic three-dimensional space and some facts about existence and volumes of polyhedra in hyperbolic space. We present fullerenes as bounded right-angled (with all dihedral and planar angles equal $\pi/2$) hyperbolic polyhedra and demonstrate new upper bounds for the volume of hyperbolic fullerenes, see Theorems~\ref{theorem34} and \ref{theorem35}. In Section~\ref{sec4} we discuss some topological indices of fullerene graphs, as Wiener index, pentagon adjacency index and others. In Section~\ref{sec5} we obtain the correlation of hyperbolic volume with topological indices and, so, with chemical properties of fullerenes, see tables therein. It demonstrates that hyperbolic volume can serve as a fullerene stability descriptor. In the last Section~\ref{sec6}  we formulate few conjectures on volumes of hyperbolic polyhedra of fullerene shapes. These conjectures are based on the correlation between hyperbolic volumes and  the Wiener index and are confirmed for the initial list of fullerenes. 

\section{Fullerenes} \label{sec2} 

Fullerene is a spherically shaped molecule built entirely from carbon atoms. The carbon atoms form rings, each of which is either a pentagon or a hexagon. Each atom has bonded with exactly three other atoms. The first fullerene, known as a "Buckminsterfullerene", was discovered in 1985, see~\cite{Kroto1985}. 

Molecular graphs of fullerenes are called \emph{fullerene graphs}. Fullerene graphs are cubic, 3-connected, planar graphs with only pentagonal and hexagonal faces. The survey on fullerene graphs is presented in~\cite{Andova2016}.  It is well-known that a graph $G$ can be realized by the vertices and edges of a convex polyhedron $P$ in Euclidean space $\mathbb E^{3}$ if and only if $G$ is a 3-connected planar graph. Denote by $f_{n}$ the number of $n$-gonal faces in $P$. H.\,S.\,M~Coxeter asked whether for every $n \neq 1$ there exist polyhedra $P_{n}$ with $f_{5} = 12$, $f_{6}=n$, and $f_{k} =0$ for $k \neq 5, 6$. The affirmative answer was done in~\cite{Grunbaum1963}.  
Therefore, fullerene graphs exist for $N=20$ and all even $N \geq 24$. A computer program \emph{Buckygen}~\cite{Brinkmann2012, Goedgebeur2015} is very useful software for generating of fullerene graphs. 
Denote by $C_{N}$ a fullerene graph with $N$ vertices. All fullerenes $C_{N}$ with $N \leq 400$ have been generated~\cite{Brinkmann1997, Brinkmann2012}. Planar codes of fullerene graphs up to $130$ vertices are posted in~\cite{HoG}. The number of fullerene isomers grows very fast~\cite{{Cioslowski2014}}. For example, there exist only one fullerene isomer for $N = 20$ which is dodecahedron, same for $N = 24$ and $N=26$.  But for $N = 60$ there are $1812$ fullerene isomers~\cite{Brinkmann2012}. 

It follows from the Euler's Polyhedron Formula that any fullerene graph has exactly $12$ pentagonal faces. The total number $f$ of faces in a fullerene graph with $N$ vertices is equal to $N/2 + 2$. Therefore, the number of hexagonal faces is exactly $N/2 - 10$. Fullerenes in which no two pentagons are adjacent are called \emph{IPR} (the Isolated Pentagon Rule) fullerenes. Fullerenes of this class are deemed to be stable, see~\cite{Kroto1987}. Buckminsterfullerene has $60$ carbon atoms and it is the smallest fullerene in which all pentagons are isolated. Next examples of IPR fullerenes appears only when the number of atoms reaches $70$. Graphs of IPR fullerenes can be generated by computer program Buckygen, and the number of all non-isomorphic fullerenes and IPR fullerenes can be found in~\cite{Brinkmann2012, Goedgebeur2015}.  A topology of fullerenes was discussed in~~\cite{Schwerdtfeger2015}. 

\section{Fullerenes as Hyperbolic Polyhedra} \label{sec3} 

We observe that fullerenes can be considered from the view of Hyperbolic geometry as bounded right-angled hyperbolic polyhedra. 

Three-dimensional hyperbolic space is the three-dimensional connected and simply connected Riemannian manifold with constant sectional curvature equals to $-1$~\cite{Ratcliffe, Milnor1982}. Hyperbolic space has few specific models with their own advantages and disadvantages. 
The conformal ball model $\mathbb B^3$ of hyperbolic space is given by the unit ball 
$$
\mathbb B^{3} = \left\{ x=(x_{1},x_{2},x_{3}) \in \mathbb R^{3} \, | \, \|x \| < 1\right\}
$$ 
under the metric
$$
{ds}^2 = 4\frac{{dx_1}^2 + {dx_2}^2 + {dx_3}^2}{{(1-{\|x\|}^2)}^2}.
$$
Geodesics in $\mathbb B^3$ are either line segments through the origin or arcs of circles orthogonal to $\partial  \mathbb B^3$. The totally geodesic subspaces of $\mathbb B^3$  are the intersections with $\mathbb B^3$ of generalized spheres (spheres or hyperplanes) orthogonal to $\partial  \mathbb B^3$.
The upper half-space model $\mathbb U^{3}$ of hyperbolic 3-space can be defined as the set 
$$
\mathbb U^{3} = \left\{ (x+iy,t) \in \mathbb C  \times \mathbb R  \, | \,  t > 0 \right\} 
$$ 
under the metric
$$
{ds}^2 = \frac{{dx}^2 + {dy}^2 + {dt}^2}{t^2}.
$$
Geodesics in $\mathbb U^{3}$ are vertical lines and semicircles which are orthogonal to the boundary $\partial \mathbb U^3 = \mathbb C \cup   \left\{ \infty \right\}  $. Totally geodesic planes are vertical planes and hemispheres centered on $\mathbb C$ . 

Let $\{ H_i, \, i = 0, \ldots, n\}$ be a finite collection of hyperbolic closed half-spaces. A \emph{hyperbolic polyhedron} is an intersection 
$
P = \bigcap^n_{i=0} H_i
$
having non-empty interior
A hyperbolic polyhedron is said to be \emph{bounded}  if all its vertices belong to the interior of a hyperbolic space.

A graph is said to be  \emph{cyclically k-connected} if at least $k$ edges have to be removed to split it into two connected components both having a cycle. A polyhedron is called \emph{acute-angled} if all its dihedral angles are at most $\pi/2$. The following type of rigidity holds in a hyperbolic space. 

\begin{theorem}  \cite{Andreev1970}
A bounded acute-angled polyhedron in an $n$-dimensional hyperbolic space, $n \geq 3$, is uniquely  (up to isometry) determined by its combinatorial type and dihedral angles. 
\end{theorem}

Therefore, one can fix a 1-skeleton of a polyhedron and check how its geometry changes under variation of dihedral angles. We are interested in another situation, namely, let us fix all dihedral angles to be equal to $\pi/2$. Then geometry of a polyhedron, say its volume, will be uniquely determined by its combinatorics, i.e. by its 1-skeleton that is a graph. Below we will use this idea for fullerene graphs. Assuming that fullerene is realized in a hyperbolic 3-space with all dihedral angles $\pi/2$, we will correspond to a fullerene the \emph{hyperbolic volume} which already depends on fullerene combinatorics only. 

The necessary and sufficient conditions for a polyhedral graph to be realized as a bounded right-angled (with all dihedral angles equal $\pi/2$) polyhedron were obtained by A.\,V.~Pogorelov~\cite{Pogorelov1967} in 1967 and E.\,M.~Andreev~\cite{Andreev1970} in 1970. These conditions can be formulated in the following form. 

\begin{theorem}  \cite{Pogorelov1967, Andreev1970}  \label{Andreev} 
A polyhedral graph is 1-skeleton of a bounded right-angled hyperbolic polyhedron if and only if the graph is 3-regular and cyclically 5-connected.
\end{theorem}

The class of bounded right-angled hyperbolic polyhedra has many nice properties and can be used to construct hyperbolic 3-manifolds by four-colorings of faces~\cite{Vesnin2017}. 

In~\cite{Doslic2001} T. Do{\v s}li{\' c} observed that every fullerene graph is cyclically 5-connected. So the following corollary holds.

\begin{corollary}
Any fullerene graph can be realized as a bounded right-angled polyhedron in three-dimensional hyperbolic space.
\end{corollary}

In Fig.~\ref{fig1} there are presented two isomers of $C_{48}$ seen as right-angled hyperbolic polyhedra. They are drawn in the conformal ball model $\mathbb B^{3}$ using the computer scripts \emph{Hyper-hedron}~\cite{Roeder2007} and the computer program \emph{Geomview}~\cite{Geomview}. 
  \begin{figure}[h]
	\centering
	\includegraphics[width=12.5cm]{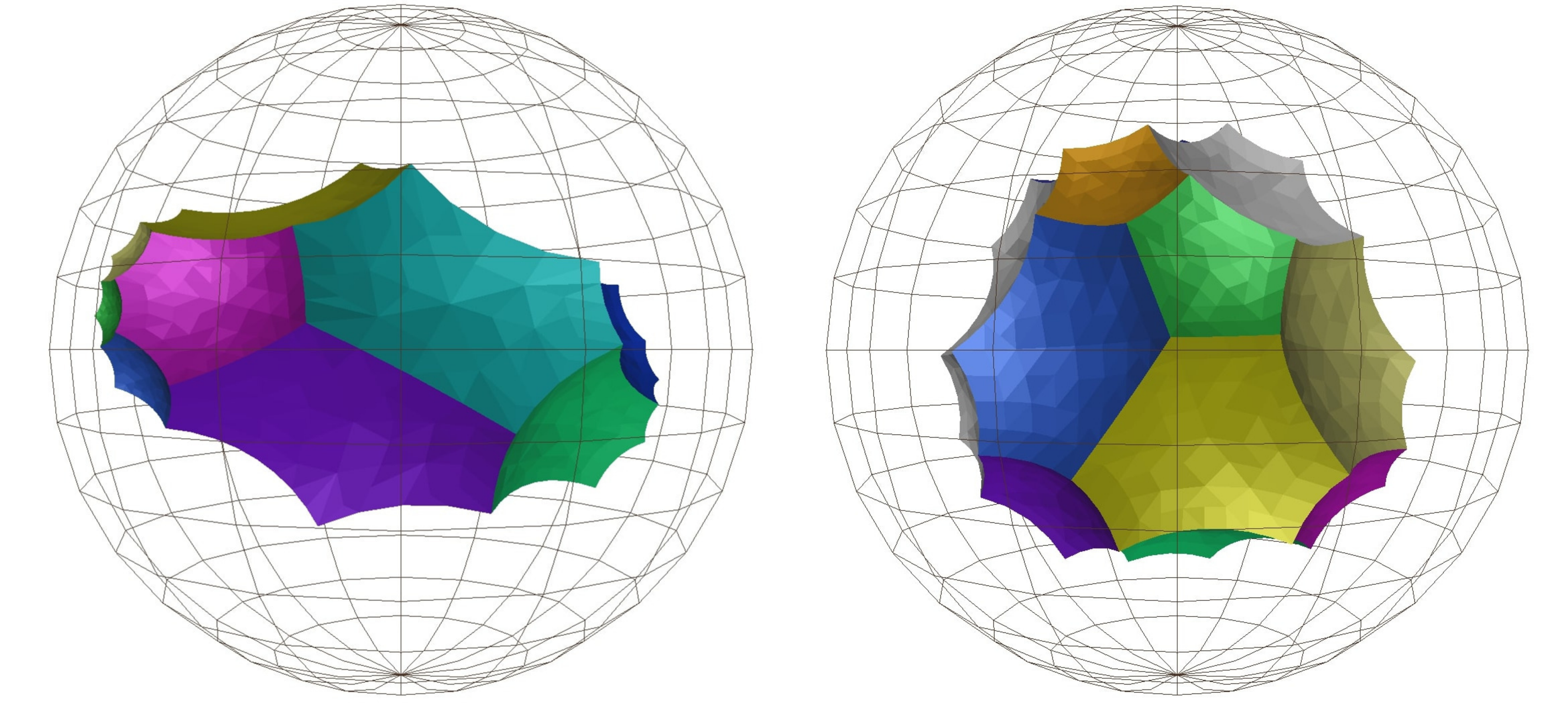} 
	\caption{Two isomers of $C_{48}$  as right-angled polyhedra in hyperbolic space.}  
	\label{fig1} 
\end{figure}
Usually, volumes of hyperbolic polyhedra are computed in terms of function
$$
\Lambda(\theta) = - \int\limits_0^{\theta} \log | 2 \sin (t) | \, {\rm d} t  \label{lob},
$$
which was introduced by J.~Milnor in the survey paper~\cite{Milnor1982} and called the \emph{Lobachevsky function}. Here and below we will give approximate values of volumes, usually up to six decimal places.  
The fullerene, presented on the left-hand side in Fig.~\ref{fig1}, has hyperbolic volume~$17.034558$ that is minimal among all $C_{48}$ fullerenes, and the fullerene, presented on the right-hand side in Fig.~\ref{fig1}, has hyperbolic volume~$18.617604$ that is maximal. 

In some cases, not only approximate values of volumes but also, explicit formula of volumes can be found. Let us consider the fullerene $C_{20}$ having the smallest number of vertices, which is a dodecahedron presented in Fig.~\ref{fig2} as a right-angled hyperbolic polyhedron. The hyperbolic volume of fullerene $C_{20}$ can be represented as follows~\cite{Vesnin1998}:  
$$
\operatorname{vol} (C_{20})=\frac{5}{2}(2\Lambda(\theta)+\Lambda(\theta+\frac{\pi}{5})+\Lambda(\theta+\frac{\pi}{5})+ \Lambda(\frac{\pi}{2}-2\theta)),
$$
where $\theta=\frac{\pi}{2}-\operatorname{arccos}(\frac{1}{2}\cos(\frac{\pi}{5}))$. This polyhedron is also called \emph{right-angled hyperbolic dodecahedron}. Numerically its volume is equal to $4,306208$ and this the smallest among volumes of all bounded right-angled hyperbolic polyhedra, as well as among all fullerenes realized as right-angled hyperbolic polyhedra.
  \begin{figure}[h]
	\centering
	\includegraphics[width=7cm]{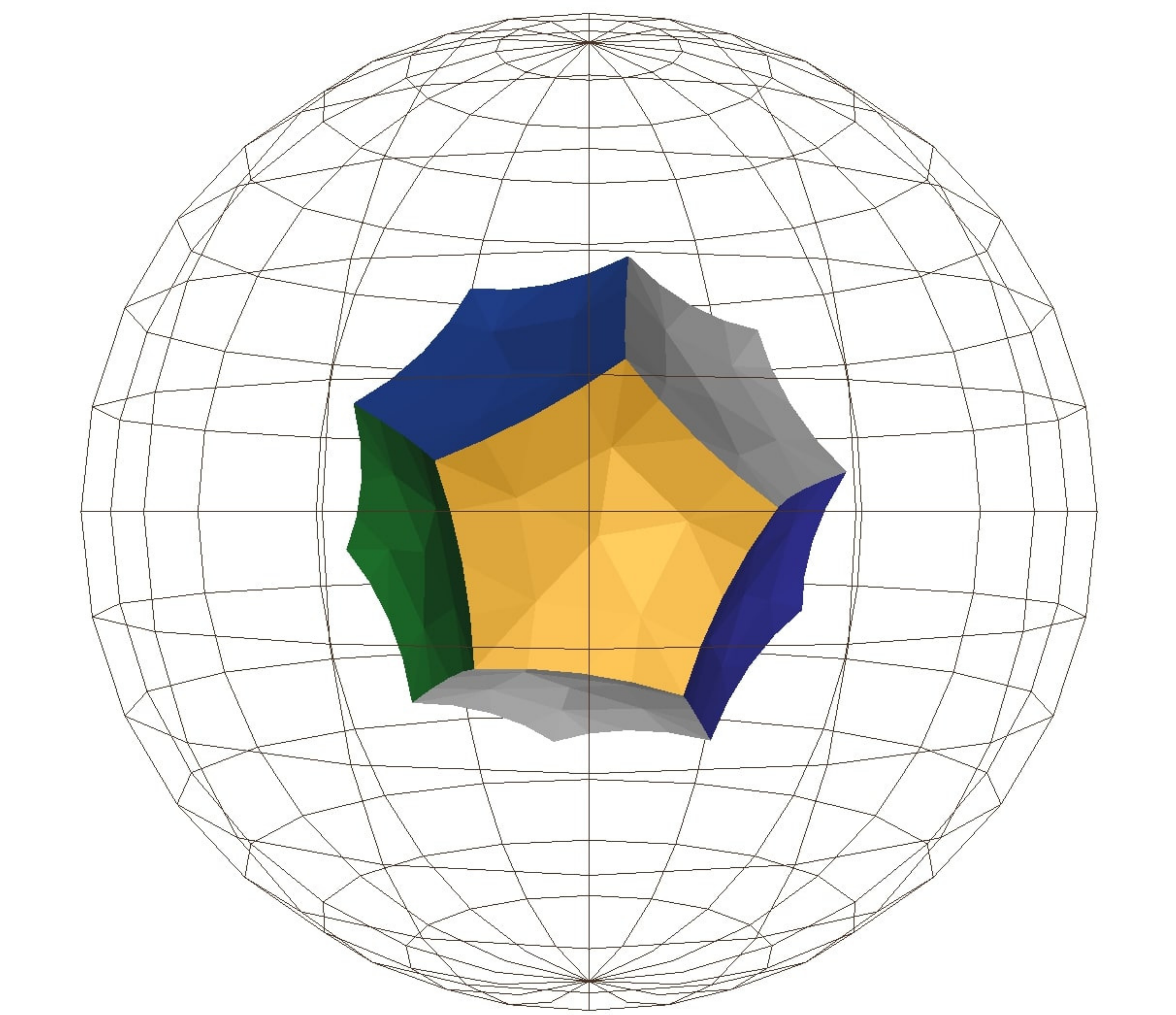} 
	\caption{Fullerene $C_{20}$ as a right-angled hyperbolic dodecahedron.}  
	\label{fig2} 
\end{figure}

Since geometry of right-angled hyperbolic polyhedron is determined by its combinatorial structure, it is natural to ask how its volume depends on some combinatorial parameters. Bilateral bounds for volumes in terms of the number of vertices were obtained by C.~Atkinson in~\cite{Atkinson2009}.

\begin{theorem} \cite{Atkinson2009}  \label{theoremAtkinson}
Let $P$ be a bounded right-angled hyperbolic polyhedron with $N$ vertices, then
$$
(N-2) \cdot \frac{v_8}{32} \leq \operatorname{vol} (P) \leq (N-10) \cdot \frac{5v_3}{8}, 
$$
where $v_{3}$ and $v_{8}$ are the following constants: 
$$
v_3 = 3\Lambda \left(\frac{\pi}{3} \right)  = 1.014941  \quad \text{\rm and} \quad  
v_8 = 8\Lambda\left(\frac{\pi}{4} \right)  = 3.663862.
$$
\end{theorem}

An upper volume bound in Theorem~\ref{theoremAtkinson} can be improved as in following statements. For the proofs we use ideas analogous to~\cite{Vesnin2020}.  

\begin{theorem} \label{theorem34} 
Let $P$ be a bounded right-angled hyperbolic polyhedron with $N$ vertices. Let $F_{1}$ and $F_ {2}$ be two faces of $P$ such that $F_ {1}$ is $n_ {1}$-gonal and $F_ {2}$ is $n_ {2}$-gonal. Then 
$$
\operatorname{vol} (P) \leq (N-n_{1}-n_{2}) \cdot \frac{5v_3}{8}.
$$
\end{theorem}

\begin{proof}
1) Assume that faces $F_ {1}$ and $F_ {2}$ are not adjacent. We construct a family of right-angled polyhedra by induction, attaching on  each step a copy of the polyhedron~$P$. Put $P_ {1} = P$. Define $P_{2}=P_ {1} \cup_{F_{1}}P_{1}$ by identifying two copies of the polyhedron $P_ {1}$ along the face $F_ {1}$. By Theorem~\ref{Andreev} $P_ {2}$ is a bounded right-angled polyhedron. It has $N_ {2} = 2N-2n_ {1}$ vertices. Indeed, faces of $P_{1}$ adjacent to $F_{1}$ form dihedral angles $\pi/2$ with $F_{1}$, so vertices of $F_{1}$ aren't vertices in the glued polyhedron $P_{2}=P_ {1} \cup_{F_{1}}P_{1}$. For the volumes we have $\operatorname{vol} (P_{2}) = 2 \operatorname{vol} (P)$. The polyhedron $P_{2}$ has at least one face isometric to $F_{2}$. Attaching the polyhedron $P$ to the polyhedron $P_{2}$ along this face we get  $P_{3} = P_{2} \cup_{F_{2}} P = P \cup_{F_{1}} P \cup_{F_{2}} P$. Evidently, $P_{3}$ is a bounded right-angled polyhedron with $N_{3} = 3N - 2n_{1} - 2n_{2}$ vertices and volume $\operatorname{vol} (P_{3}) = 3 \operatorname{vol} (P)$. Continuing the process of attaching  $P$ alternately through the faces isometric to $F_{1}$ and $F_{2}$, we get the polyhedron $P_{2k+1} = P_{2k-1} \cup_{F_{1}} P \cup_{F_{2}} P$, which is a bounded right-angled polyhedron with $N_{2k+1} = (2k+1) N - 2 k n_{1} - 2 k n_{2}$ vertices and $\operatorname{vol} (P_{2k+1}) = (2k+1) \operatorname{vol} (P)$. Now let us apply the upper bound from Theorem~\ref{theoremAtkinson} to polyhedron $P_{2k+1}$:
$$
(2k+1) \operatorname{vol} (P) \leq \left( (2k+1) N - 2 k n_{1} - 2 k n_{2} - 10 \right) \frac{5 v_{3}}{8}.  
$$
Dividing both sides of the inequality by $(2k+1)$ and passing to the limit as $k \to \infty $, we obtain the required inequality.
	
2) Assume that faces $F_{1}$ and $F_{2}$ are adjacent. Put $P_{2}=P \cup_{F_{1}} P$. The obtained polyhedron $P_{2}$ has $N_{2}=2N-2n_{1}$ vertices and its volume is twice volume of $P$. By construction, the polyhedron $P_{2}$ has a face $F_{21}$, which is constructed from two copies of $F_{2}$ and is a $(2n_{2}-4)$-gon. Since the face $F_1$ has at least 5 sides, there is a face in $P$ adjacent to $F_{1}$, but not adjacent to $F_{2}$. As a result of attaching $P$ along $F_{1}$, this face will turn into a face $F_{22}$ in a polyhedron $P_{2}$ that has at least $6$ sides.  Thus, in $P_{2}$ there is a pair of non-adjacent faces $F_{21}$ and $F_{22} $, each of which has at least $6$ sides. This situation corresponds to the already proved case (1). Thus, for the polyhedron $P_{2}$ and its non-adjacent faces $F_{21}$ and $F_{22}$ we get:
$$
2 \operatorname{vol} (P) \leq \left( N_{2} - (2n_{2} - 4) - 6 \right) \frac{5 v_{3}}{8},
$$
from where we obtain
$$
2 \operatorname{vol} (P) \leq \left( 2N-2n_{1}-2n_{2} - 2 \right) \frac{5 v_{3}}{8}
$$
and hence, 
$$
\operatorname{vol} (P) < \left( N-n_{1}-n_2 \right) \frac{5 v_{3}}{8}. 
$$
The theorem is proved.
\end{proof}	

\begin{corollary}  \label{corollary32} 
Let $P$ be a bounded right-angled hyperbolic polyhedron with $N$ vertices. Let $F_{1}$, $F_ {2}$ and $F_ {3}$ be three faces of $P$ such that $F_ {2}$ is adjacent to both $F_{1}$ and $F_{3}$, and  $F_ {i}$ is $n_ {i}$-gonal for $i=1,2,3$.
Then
$$
\operatorname{vol} (P) \leq (N-n_{1}-n_{2}-n_{3}+4) \cdot \frac{5v_3}{8}.
$$
\end{corollary}

\begin{proof}
Let us consider polyhedron $P_{1}=P \cup_{F_{2}} P$. It has $N_1=2N-2n_{2}$ vertices and its volume is twice volume of $P$. Remark that  $P_{1}$ has a face $F_{11}$ which is a union of two copies of $F_{1}$ and has $2n_1-4$ vertices. Similar, $P_{1}$ has a face $F_{13}$ which is a union of two copies of $F_{3}$ and has $2n_3-4$ vertices.  Applying  Theorem~\ref{theorem34} to polyhedron $P_{1}$ and its faces $F_{11}$ and $F_{13}$. We get
$$
2\operatorname{vol} (P) = \operatorname{vol} (P_1) \leq (N_1 - (2n_1-4)-(2n_3-4)) \cdot \frac{5v_3}{8}.
$$
Using formula for $N_1$ and dividing both sides of the inequality by $2$ we get the result.  
\end{proof}	

These new upper bounds are working best for polyhedra having large faces. But even applying them to polyhedra with only 5 and 6-gons one can obtain a slight improvement.

\begin{theorem}  \label{theorem35} 
Let $P$ be a bounded right-angled hyperbolic polyhedron with $N \geq 24$ vertices, such that 1-skeleton of $P$ is a fullerene graph. Then  
$$
(N-2) \cdot \frac{v_8}{32} \leq \operatorname{vol} (P) \leq (N-14) \cdot \frac{5v_3}{8}.
$$
\end{theorem}	
	
\begin{proof}
A lower bound is the same as in Theorem~\ref{theoremAtkinson}. So, we need to prove only an upper bound. 

Let us firstly observe that in the polyhedron $P$ there are three hexagonal faces $F_{1}$, $F_{2}$, $F_{3}$ such that $F_{2}$ is adjacent to both $F_{1}$ and $F_{3}$. Suppose in contradiction that there is no such triple of faces. Then each hexagonal face is adjacent to at most one hexagonal face. If a hexagonal face has no adjacent hexagons (we will say that it is isolated), then through its six sides it is adjacent to the pentagonal faces. If two hexagons are adjacent to each other and none of them is adjacent to another hexagon (we will say that the faces form a pair), then their union is adjacent through ten sides with pentagonal faces. Therefore, if there are $k_{1}$ isolated hexagonal faces and $k_{2}$ pairs of hexagonal faces, then through their sides they are adjacent to pentagonal faces through $6 k_{1} + 10 k_{2}$ sides. Since 1-skeleton of $P$ is a fullerene graph the number of pentagons in $P$ is equal to $12$. Their total number of sides is $60$. If $P$ has $N \geq 46$ vertices, then, by $f = N/2 + 2$, it has $f \geq 25$ faces, and twelve of them are pentagons.  Then $k_{1} + 2 k_{2} \geq  13$, and hence $2k_{2} \geq 13 - k_{1}$. Therefore, $6 k_{1} + 10 k_{2} \geq 6 k_{1} + 5 (13 - k_{1}) = 65 + k_{1} > 60$. Thus, the number of pentagons is not enough. This contradiction implies that there is a triple of hexagonal faces $F_{1}$, $F_{2}$, $F_{3}$, where $F_{2}$ is adjacent to $F_{1}$ and $F_{3}$. By applying Corollary~\ref{corollary32} to polyhedron $P$ and its faces $F_{1}$, $F_{2}$, $F_{3}$, we get the required inequality for $N \geq 46$. 

For $24 \leq N \leq 46$, an inequality holds by direct computation, see Table~\ref{table1} below, where for each $N$ minimal and maximal volumes are presented. 
\end{proof}	

The elements of the class of $C_N$ fullerenes can be ordered in different ways. One way is to use some fullerene features, such as topological indices. Often these orderings are supposed to help to distinguish more stable fullerenes from less stable. Another way is to use hyperbolic geometry. Since each fullerene  can be realized as a bounded right-angled hyperbolic polyhedron, we can match each fullerene with the hyperbolic volume of the corresponding hyperbolic polyhedron. Thus, we can order isomers of $C_N$ by hyperbolic volume. Denote by $|C_{N}|$ the number of isomers of $C_{N}$ and by $hv_i^N$, $1 \leq i \leq |C_N|$, the $i$-th element of this ordering, where we assume that $hv_{i-1}^N \leq hv_{i}^N$. We will look at how this ordering is related to  orderings by some topological indices and stability of fullerenes. 

\begin{table}[!ht] \caption{Minimal and maximal hyperbolic volumes of fullerenes.}  \label{table1} 
\begin{center} 
\begin{tabular}{|l|r|r|r|r|c|} \hline
$N$ & number of  & number of  & minimal  & maximal & caps \\ 
	& isomers  & volumes & volume & volume & type \\ \hline
20 & 1 & 1 & 4.30620 & 4.30620 & - \\ \hline
22 & 0 & 0 & - & - & -\\ \hline
24 & 1 & 1 & 6.023046  & 6.023046 & -\\ \hline
26 & 1 & 1 & 6.967011  & 6.967011 & -\\ \hline
28 & 2 & 2 & 7.869948  & 8.000234 & - \\ \hline
30 & 3 & 3 &  8.612415 & 8.946606  & a \\ \hline
32 & 6 & 6 &  9.677265  & 9.977170 & b\\ \hline
34 & 6 & 6 & 10.753476 & 10.986057 & - \\ \hline
36 & 15 & 15 & 11.537549  & 12.084191& c \\ \hline
38 & 17 & 17 & 12.380366  & 13.138893 & b \\ \hline
40 & 40 & 40 & 12.918623 & 14.222648 & a  \\ \hline
42 & 45 & 44 & 14.582147 & 15.300168 & d \\ \hline 
44 & 89 & 88 & 15.084432 & 16.397833 & b\\ \hline
46 & 116 & 116& 16.489809 & 17.486616 & d \\ \hline
48 & 199 & 19& 17.034558 & 18.617604 & c \\ \hline
50 & 271 & 270& 17.224830  & 19.767011& a \\ \hline
52 & 437 & 437& 18.866563  & 20.880946 & c\\ \hline
54 & 580 & 576& 20.308384  & 22.001239& d \\ \hline
56 & 924 & 922& 20.492318 & 23.142929& b  \\ \hline
58 & 1205 & 1205 & 22.217458  & 24.326618 & d \\ \hline
60 & 1812 & 1805 & 21.531038 & 25.558609 & a \\ \hline
62 & 2385 & 2376 & 25.558609 & 26.588511& b  \\ \hline
64 & 3465 & 3455& 24.362347  & 27.763266 & c\\ \hline
\end{tabular}
\end{center}
\end{table} 

We calculated volumes of all fullerenes with at most 64 vertices. All volumes are rounded to 6 decimal places. The results of these computations are reproduced in the diagram shown in Fig.~\ref{fig.3}. From Table~\ref{table1} one can see that almost all $C_N$ fullerenes for fixed $20 \leq N \leq 64$ have different values of volume. However, some fullerenes have the same hyperbolic volume. There could be two reasons for this fact. One reason is a calculation accuracy, indeed we have only a volume value rounded to 6 decimal places. Another reason is that two fullerenes considered as right-angled hyperbolic polyhedra really can have the same volume value.  
\begin{figure}[h]
	\centering
	\includegraphics[width=13.5cm]{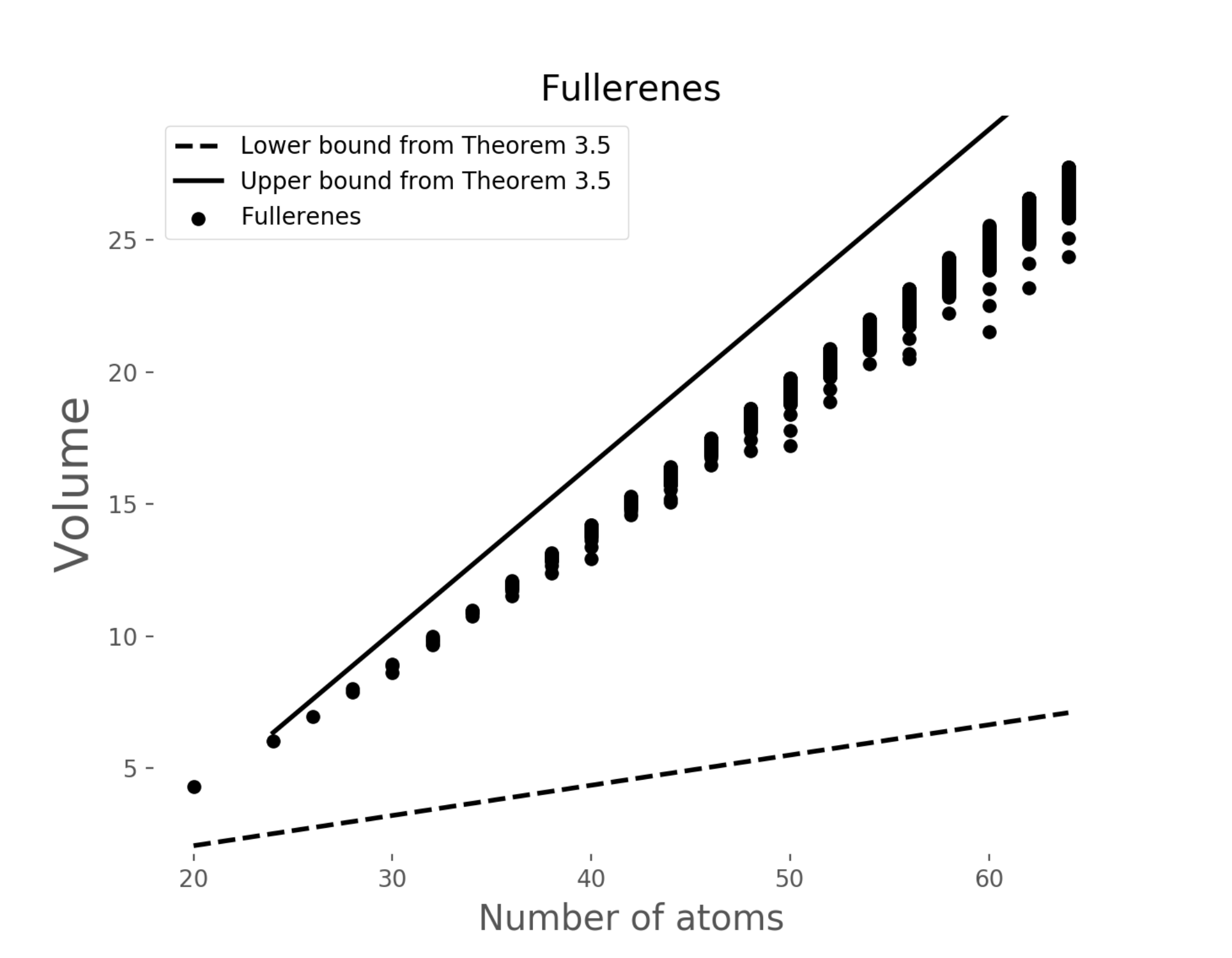} 
	\caption{Volumes of fullerenes with number of vertices from $20$ to $64$.}  
	\label{fig.3} 
\end{figure}

If two fullerenes have the same number of carbon atoms, we can compare them by their hyperbolic volume. But if two fullerenes have different numbers of atoms, then the volume of the one that has more atoms almost always is larger. 

We also suggest another way to understand how large is a hyperbolic fullerene. One can use \emph{hyperbolic sphericity}, which is a kind of normalization of volume by the surface area. Hyperbolic sphericity shows how much hyperbolic volume of fullerene is close to the volume of a hyperbolic ball of the same area. 
We recall that the volume and area of hyperbolic 3-ball $B_{r}$ of radius $r$ can be found by the following formulae:
$$ 
\operatorname{vol}(B_{r})=\pi(\operatorname{sinh}(2r)-2r) \qquad  \text{and} \qquad \operatorname{area}(B_{r})=2\pi(\operatorname{cosh}(2r)-1).
$$
Hyperbolic sphericity of a hyperbolic polyhedron $P$ is defined as follows
$$ 
\operatorname{Sp} (P)=\frac{\operatorname{vol}(P)}{\pi(\sinh (\operatorname{arccosh}(\frac{\operatorname{area}(P)}{2\pi}+1)) - \operatorname{arccosh} (\frac{\operatorname{area}(P)}{2\pi}+1))},
$$
where $\operatorname{area}(P)$ is the surface area of $P$. It is easy to see that $\operatorname{Sp}(B_{r})=1$ for any radius $r$. If $P$ is not a hyperbolic ball then $0 \leq \operatorname{Sp}(P) < 1$. If $P$ is a bounded right-angled hyperbolic polyhedron with $N$ vertices, then  $\operatorname{area}(P)=\pi(\frac{N+4}{2}-6)$ and 
$$ 
\operatorname{Sp} (P )=\frac{\operatorname{vol} (P)}{\pi( \sinh(\operatorname{arccosh} (\frac{N-4}{4})) - \operatorname{arccosh} (\frac{N-4}{4}))}.
$$
Here the denominator depends only on the number of vertices. Therefore, if we fix the number of vertices then sphericity is monotonously dependent on volume. Consequently, in a class of $C_N$ fullerenes the  orderings by hyperbolic volume and by hyperbolic sphericity are the same.

Above we considered fullerenes as bounded hyperbolic polyhedra with all dihedral (and all planar) angles equal $\pi / 2$. These angled were chosen for our convenience. Remark that by Andreev's theorem~\cite{Andreev1970} any fullerene graph can be realized as 1-skeleton of a bounded hyperbolic polyhedron with dihedral angles $\alpha_i$ satisfying $\pi/3  < \alpha_i \leq \pi / 2$. If some angles are less or equal to $\pi/3$, then the condition is: for any vertex of the polyhedron, the sum of the dihedral angles corresponding to the edges incident this vertex must be greater than $\pi$. That is, for any vertex $v_i$ the inequality holds: $\alpha_{i_{1}}+\alpha_{i_{2}}+\alpha_{i_{3}} > \pi$, where $\alpha_{i_{1}}, \alpha_{i_{2}}, \alpha_{i_{3}}$ are dihedral angles at edges incident to $v_i$. Replacing the inequality by equality $\alpha_{i1}+\alpha_{i2}+\alpha_{i3} = \pi$, we arrive to the notion of an \emph{ideal} hyperbolic polyhedron, which is a polyhedron with all vertices on the boundary of the hyperbolic space. In this case we can consider fullerenes as ideal hyperbolic polyhedra, see Rivin's theorem~\cite{Rivin1996}. 
As an example, let us consider fullerenes as ideal $\pi/3$-equiangular hyperbolic polyhedra. As well as volume bounds for right-angled hyperbolic polyhedra, C.~Atkinson~\cite{Atkinson2009} proved upper and lower volume bounds for the ideal polyhedra with all dihedral angles equal to~$\pi/3$. 

\begin{theorem} \cite{Atkinson2009} \label{theorem36}
Let $P$ be an ideal $\pi/3$-equiangular polyhedron with $N>4$ vertices, then
$$
N \cdot \frac{3v_3}{8} \leq \operatorname{vol} (P) \leq (3N-14) \cdot \frac{v_3}{2}.
$$
\end{theorem}

For fullerenes, the lower bound in Theorem~\ref{theorem36} can be improved a bit as follows. 

\begin{proposition}  
Let $P$ be an ideal $\pi/3$-equiangular hyperbolic polyhedron with $N \geq 20$ vertices, such that 1-skeleton of $P$ is a fullerene graph. Then  
$$
v_3 \cdot \ceil{\frac{N}{2}-\sqrt{\frac{3}{5}N}} \leq \operatorname{vol} (P) 
$$
\end{proposition}	

\begin{proof}
This proof is just a reiteration of the one described in~\cite{Atkinson2009}. The difference is in using more precise lower bound for an independence number.

For any given vertex $v \in P$, $v$ along with the three adjacent vertices of $P$ are the vertices of an ideal tetrahedron with all dihedral angles equal $\pi/3$. The tetrahedron is contained in $P$. If we take a collection of such tetrahedra with disjoint interiors, then the sum of their volumes will be not more than the volume of $P$. Such a collection may be constructed by taking an independent set of vertices of $P$. Denote by $V(P)$ the set of vertices of $P$. A set $V_{1} \subset V(P)$ is \emph{independent}  if no two vertices from $V_1$ are adjacent in $P$. The cardinality of any largest independent set in $P$ is called the \emph{independence number} of $P$ and denoted by $\alpha(P)$. It is known~\cite{Faria2012} that a fullerene graph $P$ with $N$ vertices contains an independent set of cardinality at least $\alpha_{P} \geq \frac{N}{2}-\sqrt{\frac{3}{5}N}$. Applying this result, we obtain the required inequality.
\end{proof}

One can expect that the consideration of the volumes of fullerenes as hyperbolic polyhedra with dihedral angles differ from $\pi/2$ will give the similar orderings of fullerenes. For instance, the Pearson correlation coefficient (PCC) of orderings by hyperbolic volumes of $C_{60}$ fullerenes considered as hyperbolic polyhedra with all dihedral angles $\pi/2$ and with all dihedral angles $\pi/3$ is equal to $0.999416$.

\section{Topological Indices} \label{sec4} 

In general, a topological index is a numerical invariant of a chemical graph. Topological indices are widely used in chemical graph theory, molecular topology, and mathematical chemistry as molecular descriptors. We will start with one of the oldest topological indices known as Wiener index. 

\subsection{Wiener index and its generalizations} 

Wiener index was introduced in 1947 by American chemist Harry Wiener \cite{Wiener1947}. After that Wiener index has found a lot of application in chemestry, see \cite{Dobrynin2001, Todeschini2000, Knor2016}. Let $G$ be a graph, $V(G)$ be the set of all vertices of $G$. Denote by $d(u,v)$ the distance between vertices $u$ and $v$ in $G$, i. e., the number of edges in the shortest path connecting $u$ and $v$. Fix $u \in V(G)$. The sum of distances from $u$ to all the other vertices of $G$ is called \emph{transmission} of vertex $u$, $\operatorname{tr}(u)=\sum_{v \in V(G) } d(u,v)$. The number of different vertex transmissions in $G$ is called the \emph{Wiener complexity} \cite{Alizadeh2016}. If all vertices of a graph have different transmissions then the graph is called \emph{transmission irregular}. It is shown in \cite{Dobrynin2019} that there do not exist transmission irregular fullerene graphs with less than $217$ vertices.    

\emph{Wiener index} is defined as a half of the sum of all vertex transmissions 
$$
W(G) = \frac{1}{2} \sum_{u \in V(G) } \operatorname{tr}(u)  =  \frac{1}{2} \sum_{\left\{u,v\right\}  \subset V(G) } d(u,v).
$$
There are many generalization of Wiener index are known. One of them is \emph{Hyper Wiener index} defined as a linear combination of distances and their squares:  
$$
WW(G) = \frac{1}{2} \sum_{\left\{u,v\right\}  \subset V(G) }^{} (d(u,v)+{d(u,v)}^2). 
$$

Let us define another generalization of Wiener index as follows. Let  $G^{*}$ be a graph dual to a fullerene graph $G$. To vertex $v \in G^{*}$ of degree five the corresponding object $G$ is a pentagonal face.  For a fullerene graph $G$ define $W_5(G)$ as the sum of squares of all distances between vertices of degree $5$ in a dual graph $G^{*}$: 
$$
W_5 (G) = \frac{1}{2} \sum_{\tiny \begin{gathered} \left\{ u,v \right\}  \subset V(G^*) \\ \operatorname{deg}(u) = \operatorname{deg}(v)=5 \end{gathered}} {d^*(u,v)}^2,
$$
where $d^*(u, v)$ is the distance between $u$ and $v$ in $G^*$.

\subsection{Pentagon Adjacency Index} 
Let $G$ be a fullerene graph. 
Pentagon-neighbor signature is a vector $(p_0, p_1, p_2, p_3, p_4, p_5)$, where each entry $p_k$, $k=0, 1, \ldots, 5$, counts those pentagons which have exactly $k$ pentagonal neighbors. The \emph{pentagon adjacency index} $N_p$ is defined in the following way: 
$$
N_p (G) = \frac{1}{2} \sum_{k=0}^{5} k p_k, 
$$
In other words, $N_p$ is the number of pairs of adjacent pentagons in a fullerene graph.

The pentagon adjacency index $N_p$ is one of the most popular topological predictors of fullerene stability. In most cases, for the stability prediction of lower fullerene isomers $C_N$ with $N \leq 70$ the index $N_P$ is used \cite{Fowler1995, Reti2009, Albertazzi1999, Campbell1996}. When $N$ is getting larger, the number of different fullerenes having the same $N_p$ is growing fast. For instance, the number of IPR fullerenes with $120$ vertices (IPR fullerenes are exactly those fullerenes for which $N_p=0$) is equal to $10774$. The index $N_p$ does not distinguish such fullerenes. For larger $N$ we can get more information about a fullerene structure by looking at hexagons rather than pentagons.

We will also use kind of a generalization of $N_p$ which is the second moment of the pentagon neighbor signature: 
$$
H_5 (G) =\sum_{k=0}^{5} k^2 p_k.
$$

\subsection{The second moment of the hexagon neighbor signature.} 
Similarly, a hexagon-neighbor signature $(h_0, h_1, h_2, h_3, h_4, h_5, h_6)$ was defined in~\cite{Fajtlowicz2003}. Here each entry $h_k$, $k=0, 1, \ldots, 6$, counts those hexagons which have exactly $k$ hexagonal neighbors. The topological index 
$$
H_6 =\sum_{k=0}^{6} k^2 h_k
$$
is sometimes referred as a \emph{Fowler-Manolopoulos Predictor of Fullerene Stability}~\cite{Ju2010}. It has a strong correlation with the stability of fullerenes, see ~\cite{Doslic2005}.

\section{Hyperbolic Volume, Topological indices and Stability of Fullerenes} \label{sec5} 

In~\cite{Sure2017} there are presented computations of  the relative energies of all $1812$ isomers of $C_{60}$ and testing of $26$ chemical descriptors for their correlation with the energetic ordering of these isomers. The authors identify $7$ rules among $26$ which they call \emph{good stability criteria}. These criteria are defined as 
\begin{itemize}
\item[(1)] the one which can identify correctly the first two energetically most stable and the three energetically least stable isomers in the correct energetic order; 
\item[(2)] the Pearson correlation coefficient (PCC) of linear correlation between the relative energies of all $C_{60}$-isomers and their criterion values should be larger than $0.6$; 
\item[(3)] the slope and the Pearson correlation coefficient in the linear regression of relative energy vs. the criterion for $C_{60}$-isomers with $N_p \in \left\{4, \ldots, 14\right\} $ should have the same sign. 
\end{itemize} 

It is interesting that hyperbolic volume satisfies these three requirements. It can be seen from  Fig.~\ref{fig4} that there are two isomers of $C_{60}$ with the largest volume coincide with two having the smallest relative energy. And also three isomers of $C_{60}$ with the smallest volume coincide with three having the largest relative energy. The values of volume of this five fullerenes and their relative energies are shown in Table~\ref{table2}.
  \begin{figure}[h]
	\centering
	\includegraphics[width=12cm]{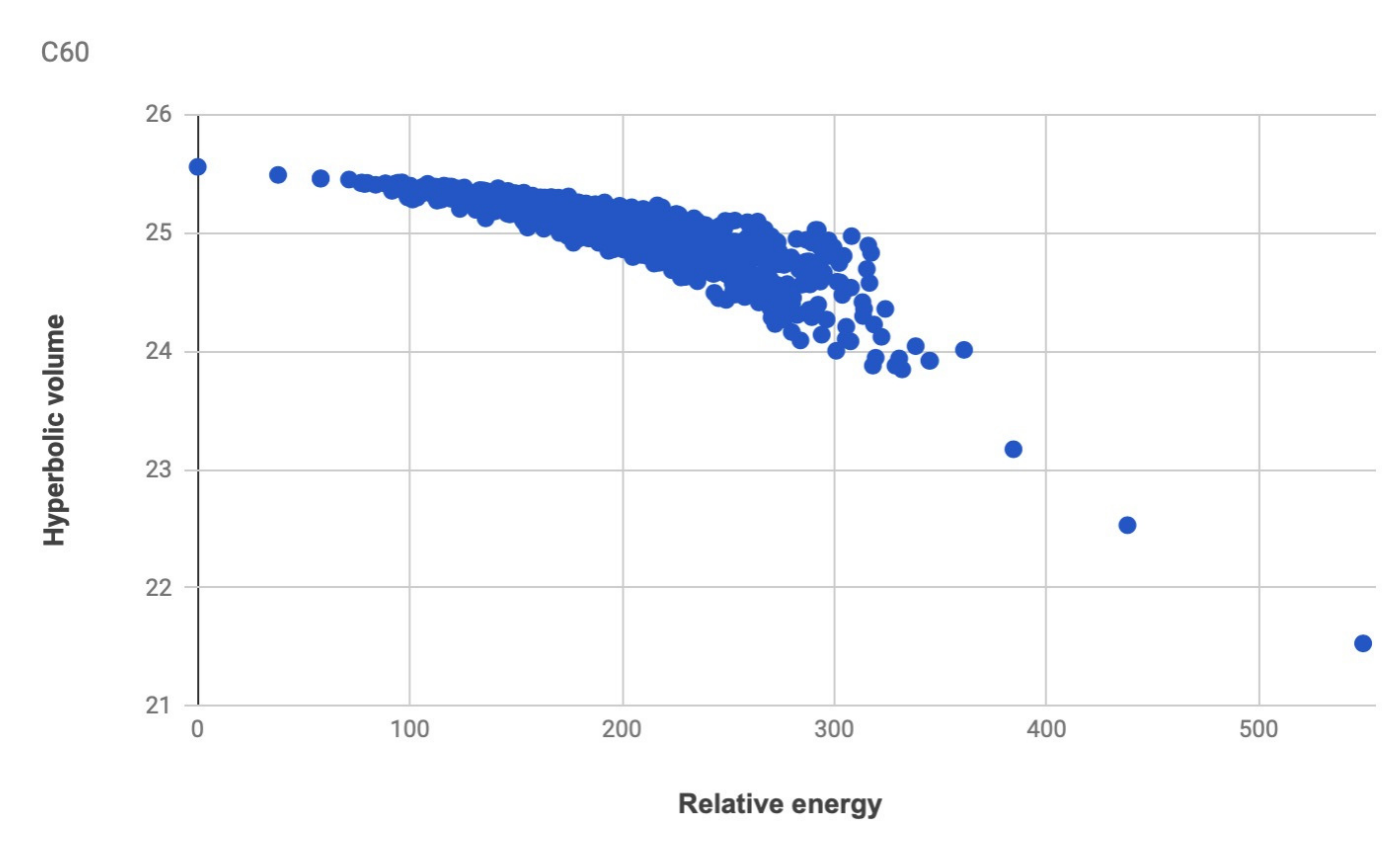} 
	\caption{Scatter chart of volume and relative energy} 
	\label{fig4} 
\end{figure}
\begin{table}[h] \caption{Three smallest and two largest vo\-lumes and corresponding relative energies.} 
	\label{table2} 
	\centerline{ 
		\begin{tabular}{|l|l|l|l|l|l|} \hline
			Isomer ordering & Volume  &  Relative energy \\ \hline 
			$hv_{1}^{60}$  & 21.531038 & 549.0652657  \\ \hline		
			$hv_{2}^{60}$  & 22.530413 & 438.02482  \\ \hline		
			$hv_{3}^{60}$  & 23.172024 & 384.307823  \\ \hline		
			$\vdots$ & $\vdots$ & $\vdots$ \\ \hline  
			$hv_{1811}^{60}$  & 25.490485 & 37.87200704 \\ \hline		
			$hv_{1812}^{60}$  & 25.558609 & 0  \\ \hline		
			\end{tabular} 
		}
		\end{table} 

It is demonstrated in  Table~\ref{table3}  that the correlation between volume and relative energy is $-0.862611$ which absolute value is bigger than $0.6$. 

\begin{table}[!ht] \caption{Pearson correlation coefficients between the hyperbolic vo\-lu\-me, topological indices and relative energy of isomers of $C_{60}$}  \label{table3} 
	\centerline{ \tiny
		\begin{tabular}{|l|l|l|l|l|l|l|l|l|} \hline
		  &Hyperbolic & $W$ & $WW$ & $W_5$ & $N_p$ & $H_5$ & $H_6$ &  Relative \\ 
		  &volume & & & & & &  & energy  \\ \hline
			Hyperbolic& 1 & -0.831511& -0.822173 & -0.960028 & -0.884936 & -0.944754 & -0.914185 & -0.862611\\ 
			volume & & & & & & & & \\ \hline
			$W$ &-0.831511&1&0.998522&0.779468&0.625317&0.680390&0.651882&0.627104\\ \hline
			$WW$ &-0.822173&0.998522&1&0.768740&0.595728&0.661271&0.625076&0.598044\\ \hline
			$W_5$ &-0.960028&0.779468&0.768740&1&0.849136&0.882059&0.899071&0.878219\\ \hline
			$N_p$ &-0.884936 &0.625317&0.595728&0.849136&1&0.961484&0.985533&0.952449\\ \hline
			$H_5$ &-0.944754&0.680390&0.661271&0.882059&0.961484&1& 0.946223&0.906894\\ \hline
			$H_6$ & -0.914185&0.651882&0.625076&0.899071&0.985533&0.946223&1&0.961179\\ \hline
			Relative &-0.862611&0.627104&0.598044&0.878219&0.952449&0.906894&0.961179&1\\ 
			energy & & & & & & & & \\ \hline
	\end{tabular} }
\end{table} 

In Table~\ref{table4} one can see that slopes and correlation coefficients for all $N_p \in \left\{4, \ldots,14\right\}$  have the same sign. So, the hyperbolic volume can be regarded as a good stability criteria.
Scatter charts of volume and some topological indices of $C_{60}$ fullerenes are demonstrated in Fig.~\ref{fig.9}, where Wiener index, $N_{p}$, Hyper-Wiener index, and $W5$ are presented. 

\begin{table}[!ht] \caption{Slopes of the linear fits and Pearson correlation coefficients of relative energy vs. hyperbolic volume for the given $N_p$.} \label{table4} 
	\centerline{ 
		\begin{tabular}{|l|r|r|} \hline
			$N_p$ & Slope \quad \qquad& PCC \quad \qquad  \\ \hline
		    $4$ & -136.446905 \qquad & -0.252837  \\ \hline
		    $5$ & -7.750791  \qquad& -0.032344  \\ \hline
		    $6$ & -15.859778  \qquad& -0.074562  \\ \hline
		    $7$ & -34.536717  \qquad& -0.143141  \\ \hline
		    $8$ & -14.999570  \qquad& -0.074147   \\ \hline
		    $9$ & -27.831100  \qquad&  -0.152313   \\ \hline
		    $10$ & -28.468661  \qquad& -0.162604  \\ \hline
		    $11$ & -29.284308  \qquad& -0.214940   \\ \hline
		    $12$ & -3.409933  \qquad& -0.046660   \\ \hline
			$13$ & -30.516286  \qquad& -0.461604   \\ \hline
			$14$ & -19.286370  \qquad& -0.317933   \\ \hline
		\end{tabular} }
\end{table} 
 \begin{figure}[ht]
	\centering
	\includegraphics[width=15.4cm]{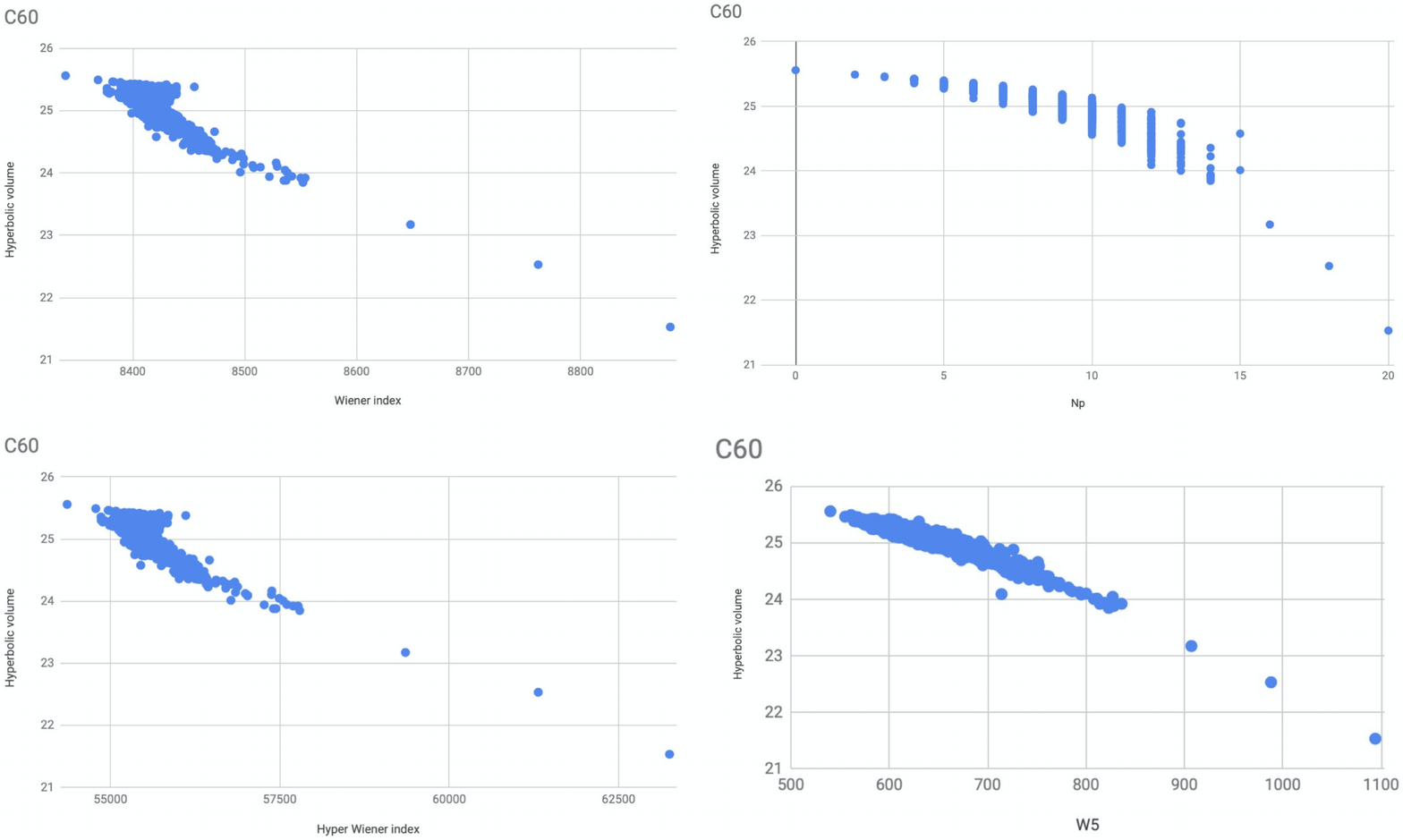} 
	\caption{Scatter charts of volume and topological indices.} 
	\label{fig.9} 
\end{figure}

Let us consider fullerenes with other numbers of atoms. Relative energies of 15 isomers of $C_{36}$ are given in ~\cite{Kim2005}. Hyperbolic volume indicates an isomer with the largest energy as the one with minimal volume $11.537549$. But an isomer with the lowest relative energy has the second largest volume $12.046092$. The maximal volume $12.084191$ has an isomer with the second smallest energy. Pearson correlation coefficient between volumes and energies is $-0.922790$.

Now let us demonstrate that hyperbolic volumes distinguish IPR fullerenes. It is shown in~\cite{Fowler2001} that in combination with the second moment of hexagon neighbor signature $H_6$, Wiener, Szeged, and Balaban indices can distinguish low-energy IPR fullerenes having  $84$ and $100$ carbon atoms. For the class of fullerenes with $100$ atoms minimization of $H_6$ gives $38$ different isomers with the same value of $H_6$, equals to $820$. Minimization of Wiener or Szeged indices, the same as maximization of Balaban index, would select the energetically unfavorable isomers. But among $38$ isomers with a minimum value of $H_6$ all three named indices select the isomer favored in quantum mechanical energy calculations,  that is an isomer D2 100:459. Let us now turn to hyperbolic volumes. Maximal hyperbolic volume of $C_{100}$ IPR fullerenes is given by the same isomer. The maximal volume is equal to $49.031808$. So this hyperbolic volume is attached to the best known best isomer among $450$ other $C_{100}$ IPR fullerenes. Similarly, in the class of fullerenes with $84$ atoms, the biggest hyperbolic volumes are attached to two isomers which are highlighted in ~\cite{Fowler2001} as the most stable ones. 

Among seven IPR fullerenes with $80$ atoms, only isomers $31918$ and $31919$ (numbers in the ordering of $C_{80}$ by lexicographically minimal spiral development) can be chemically separated, see ~\cite{Khamatgalimov2011}. The hyperbolic volume of these isomers are two smallest, Table~\ref{table5} presents  hyperbolic volumes, second moment of hexagon neighbor signatures $H_6$, nuclear volumes (we will discuss later) and relative energies of these seven isomers. It is intresting that the correlation reverses sign here. As one can see in Table~\ref{table5}, the same thing happens with $H_6$. In~\cite{Bille2019} autors notice the similar sitiation for Newton polynomials.
\begin{table}[!ht] \caption{Hyperbolic volumes, $H_6$, nuclear volumes and relative energies of seven IPR isomers of $C_{80}$.} \label{table5} 
	\centerline{ 
		\begin{tabular}{|c|c|c|c|c|} \hline
			Number & Hyperbolic volume & $H_6$ & Nuclear volume & Relative energy   \\ \hline
			$31918$ & 37.043523 & 500& 241.48452 & 0   \\ \hline
			$31919$ & 37.093128 & 496& 241.63592 & 0.41  \\ \hline
			$31921$ & 37.122213 & 492& 242.47294 & 4.37  \\ \hline
			$31920$ & 37.150100 & 488& 245.25394 & 2.58  \\ \hline
			$31922$ & 37.168408 & 484& 248.21569 & 1.48   \\ \hline
			$31923$ & 37.181925 & 480& 250.89256 & 3.32  \\ \hline
			$31924$ & 37.181925 & 480& 252.65704 & 14.31  \\ \hline
		\end{tabular} }
\end{table}

Correlations of hyperbolic volume with  topological indices $H_{5}$, $H_{6}$, $W_{5}$ and $N_{p}$ are presented in Fig.~\ref{fig.5}. The values of volumes and indices are compared for fullerenes  $C_N $ with  $48 \leq N \leq 62$ and $N=80$ atoms. It is clear from the figure that topological indices $ H_5 $ and $ W_5 $ demonstrate the best correlation with hyperbolic volume. For $W_5 $ the correlation is usually larger than $0.95$ by the absolute value. 
 \begin{figure}[ht]
	\centering
	\includegraphics[width=15.4cm]{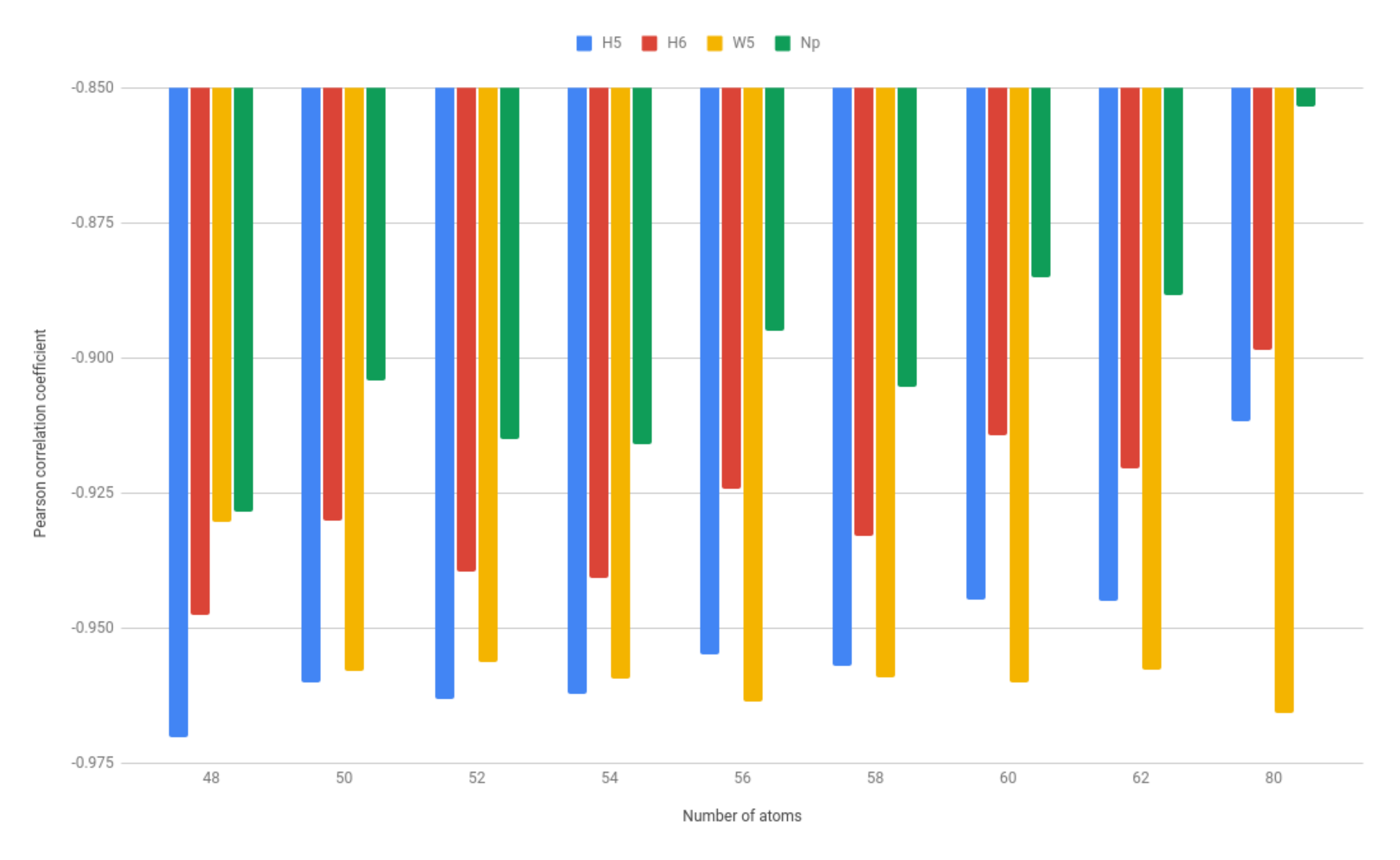} 
	\caption{Correlations between topological indices  and hyperbolic volume} 
	\label{fig.5} 
\end{figure}

Now let us turn to the case of IPR fullerenes. 
Obviously, in the case of IPR fullerene $H_5$ is always equal to zero and cannot give a good correlation with hyperbolic volume. Computations show  that $W_5$ also does not have a good correlation with hyperbolic volume. For example, Pearson correlation coefficient of $W_5$ and vo\-lu\-me of IPR fullerenes with $100$ atoms is $-0.422241$. At the same time, in the case of IPR fullerenes topological index $H_6$ shows a strong correlation with hyperbolic volume of fullerenes with $88 \leq N \leq 100$ atoms, see Table~\ref{table6}.

\begin{table}[!ht] \caption{Correlations between $H_6$ and hyperbolic volumes of IPR fullerenes} \label{table6} 
	\centerline{ \small 
		\begin{tabular}{|l|l|l|l|l|l|l|l|} \hline
			Number&88&90&92&94&96&98&100\\ 
			of atoms& &&&&&&   \\ \hline
			PCC &-0.933654&-0.929643&-0.972353&-0.946845&-0.946845&-0.939880&-0.931683 \\ 
			  & & & & & & & \\ \hline
		\end{tabular} 
	}
\end{table} 

It is interesting to use a multiple linear regression model to predict hyperbolic volume of $C_{60}$ fullerenes based on topological indices. We got $R^2$ (coefficient of determination) value equal to $0.992135$. The final predictive model was: 
\begin{displaymath}
\begin{gathered}
HV = 52.3061617-0.0102493573\cdot H_6 +  0.198026983 \cdot N_p -0.00298725582\cdot W\\ -0.00004197197\cdot W_5 -0.0178469333\cdot H_5
\end{gathered}
\end{displaymath}
The implementation of this model to all fullerene isomers with 80 atoms gave us PCC equal to $0.985944$.

Also, we build a multiple linear regression model that predicts relative energies of $C_{60}$ fullerenes based on all described topological indices and hyperbolic volume. Value of $R^2$ was $0.941725$. And the final predictive model was: 
\begin{displaymath}
\begin{gathered}
RE=-25159.684305+1.29550558\cdot H_6-2.54481396\cdot N_p+4.48823021\cdot W\\-0.26126652\cdot WW+0.411231354\cdot W_5-0.0311023629\cdot H_5+56.8461886\cdot HV
\end{gathered}
\end{displaymath}

Program Fullerene~\cite{Schwerdtfeger2013} generates cartesian coordinates for fullerene isomer. Then it can easily calculate the euclidean volume. This volume is also called the nuclear volume of an isomer~\cite{Adams1994}. It has been shown that this volume is an indicator of fullerene energies~\cite{Adams1994}. In~\cite{Sure2017} authors show that in the case of $C_{60}$ the correlation coefficient between this volume and relative energy is equal to $-0.908$. We find out that there is a strong connection between hyperbolic and nuclear volume, which can be seen in the following examples. In Table~\ref{table5} nuclear volumes of seven IPR $C_{80}$ isomers are presented. One can see that they monotonously depend on hyperbolic volume. As another example let us consider all 199 fullerene isomers with 48 atoms. From Table~\ref{table7} one can see that the correlation of hyperbolic volume and nuclear volume, in this case, is $0.84554$.

\begin{table}[!ht] \caption{Correlations of hyperbolic volume and nuclear volume with topological indices of $C_{48}$ fullerenes} \label{table7} 
	\centerline{ \small 
		\begin{tabular}{|l|l|l|l|l|l|l|} \hline
			 &$H_5$&$H_6$&$W_5$&$N_p$&Hyperbolic& Nuclear \\ 
			 &&&&&volume&volume\\ \hline
			Hyperbollic&-0.970026&-0.947454&-0.930290&-0.928317&1&0.845541 \\
			volume& &&&&&   \\ \hline
			Nuclear volume&-0.881119&-0.935015&-0.884770&-0.918476&0.845541&1 \\ \hline
		\end{tabular} 
	}
\end{table}

\section{Conjectures on volumes based on Wiener index} \label{sec6} 

In \cite{Dobrynin2019} there were calculated Wiener indices of all fullerene graphs up to $216$ vertices. The results of this calculation led authors to a conjecture that, if a fullerene graph of an arbitrary order has the maximal Wiener index, then it is a nanotubical fullerene graph with caps of four types (a), (b), (c) or (d). Four types of caps are shown in Fig.~\ref{fig.6}, see~\cite{Dobrynin2019}.  It turns that, in our calculation, fullerenes isomers, which graphs have the maximal Wiener index, have the minimal hyperbolic volume. 
\begin{figure}[ht]
	\centering
	\includegraphics[width=14.cm]{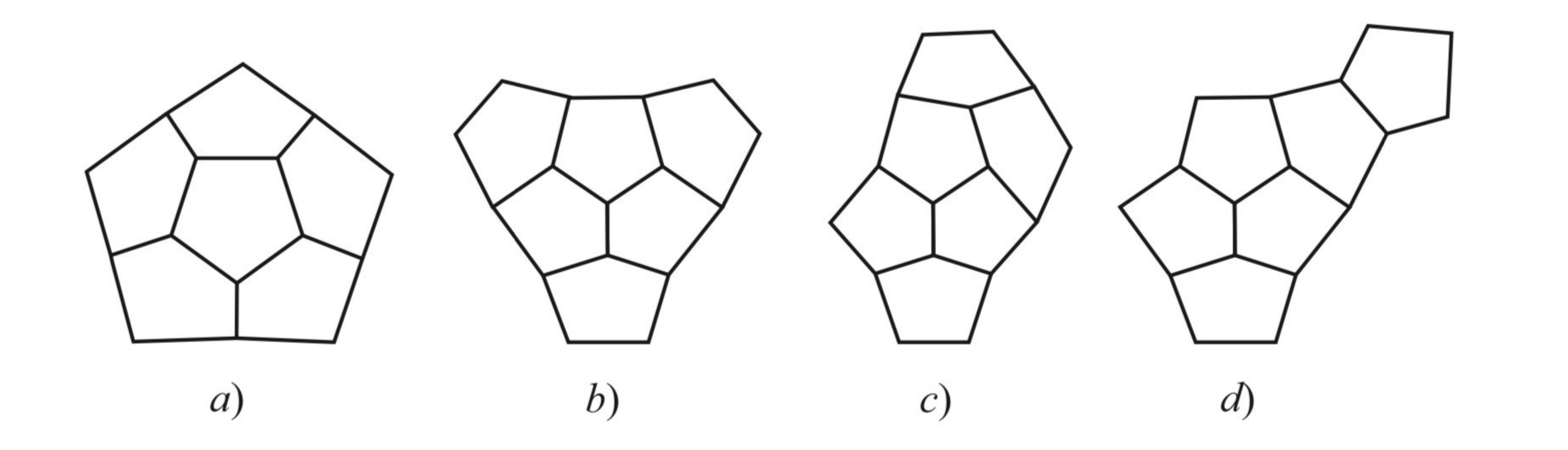} 
	\caption{Four types of caps.} 
	\label{fig.6} 
\end{figure}

It is known that a nanotubical fullerene graph with caps of type (a) exists only for a number of vertices $N = 10k$, $k \geq 2$, see ~\cite{Alizadeh2014}.

\begin{proposition} \label{prop.1}
	Let $F$ be  a nanotubical fullerene graph with caps of type (a). Then $F$ has $N = 10k$ vertices, $k \geq 2$, and 
	$$
	\operatorname{vol}(F) = \left (\frac{N}{10}-1 \right) \cdot \operatorname{vol} (D),
	$$
	where $D$ is a right-angled hyperbolic dodecahedron. Numerically $\operatorname{vol} (D)=4,30620$.
\end{proposition}

\begin{proof}
Let us draw a lateral surface of $C_{20}$ as in Fig.~\ref{fig.7}, where left and right sides are assumed to be identified. 
 \begin{figure}[ht]
		\centering
		\includegraphics[width=6.5cm]{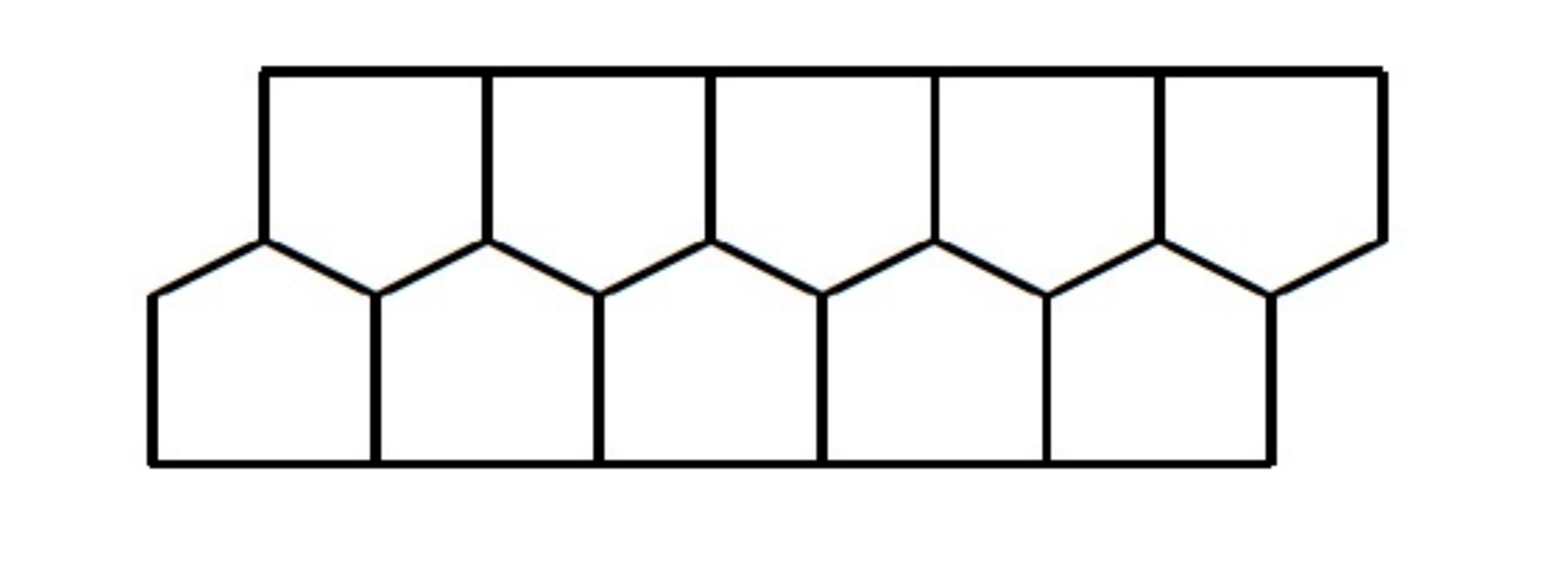} 
		\caption{The lateral surface of $C_{20}$.} 
		\label{fig.7} 
	\end{figure}
  
 Then a composition of $k$ copies of $C_{20}$ gives us a polyhedron with two caps of type (a), $N=10(k+1)$ vertices, and volume as in the formula above. Its lateral surface is presented in Fig.~\ref{fig.8}. 
 \begin{figure}[ht]
	\centering
	\includegraphics[width=8.5cm]{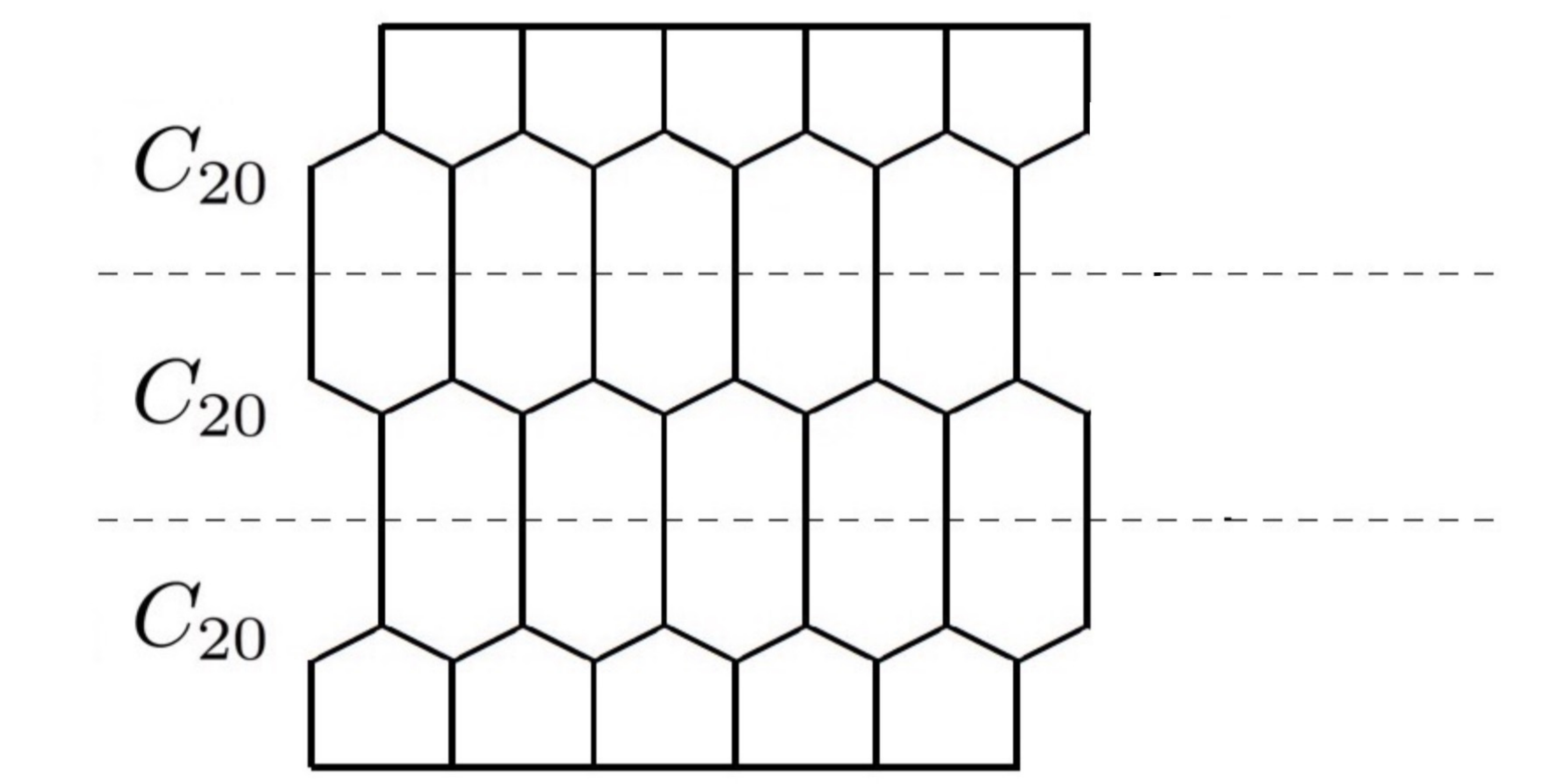} 
	\caption{The union of three copies of $C_{20}$.} 
	\label{fig.8} 
\end{figure}
 Denote this polyhedron by $C_{20}^k$. As one can see from Fig.~\ref{fig.8}, $C_{20}^3$ has two caps of type (a) and $40=10(3+1)$ vertices. Similarly,  $C_{20}^k$ for $k \geq 1$, has a lateral  surface consisting of $5(k-1)$ hexagons and two pentagonal cups of type (a).  Since $C_{20}$ can be realized as hyperbolic right-angled dodecahedron $D$, we get
	$$
	\operatorname{vol}(C_{20}^k) = k \cdot \operatorname{vol}(C_{20})= k \cdot \operatorname{vol}(D)= \left( \frac{N}{10}-1 \right) \cdot \operatorname{vol} (D). 
	$$
\end{proof}

\begin{conjecture} 
If fullerene  with $N = 10k$ atoms, $k \geq 2$, has the minimal hyperbolic volume in the class $C_{N}$, then it is nanotubical fullerene with caps of type (a) and its volume is given in Proposition~\ref{prop.1}. 
\end{conjecture}
 
It is known that a nanotubical fullerene graph with caps of type (b) has $N = 6k-4$ vertices, $k \geq 5$, see~\cite{Dobrynin2019}. 

\begin{conjecture} 
If fullerene with $N = 6k-4 \neq 10s$ atoms, $k \geq 5$, $s \geq 2$, has the minimal hyperbolic volume in the class $C_{N}$, then it is a nanotubical fullerene with caps of type (b). If fullerenes with caps of types (a) and (b) have the same number of atoms ($N = 10k$), then the fullerene with caps of type (a) has the minimal hyperbolic volume.
\end{conjecture}

For the cases of fullerenes with caps of types (c) and (d) there is no universal formula for the number of vertices. In~\cite{Dobrynin2019} fullerenes with caps of types (c) were splited into two disjoint classes:  $T_{c1}$ and $T_{c2}$. Similarly, fullerenes with caps of types $d$ were splited  into classes  $T_{d1}$ and $T_{d2}$. It is shown that for all $N \leq 216$ such that  $N \neq 20, 24, 28, 34$, $N \neq 10s$, and  $N \neq 6k-4$, where $k \geq 5$, $s \geq 2$, fullerenes with maximal Wiener index belong to $T_{c1}$, $T_{c2}$, $T_{d1}$ or $T_{d2}$.
 
\begin{conjecture} 
If a fullerene isomer with number of atoms $N \neq 20, 24, 26, 28, 34, N \neq 10s,  N \neq 6k-4$, where $k \geq 5$, $s \geq 2$, has the minimal hyperbolic volume in the class $C_{N}$, then it is is a nanotubical fullerene graph with caps of type (c) or (d). 
\end{conjecture}

Numerical computations demonstrate that above conjectures are confirmed for $N \leq 64$. Indeed, the caps types corresponding to maximal hyperbolic volume isomers are indicated in the last column of Table~\ref{table1}. 
Remark that in exceptional cases $N = 20, 24, 26, 28, 34$ fullerene with minimal hyperbolic volume has one pentagonal component~\cite{Dobrynin2019}.

\end{document}